\documentclass[11pt]{article}
\usepackage{graphicx}
\usepackage[utf8x]{inputenc}
\usepackage{amsmath,sectsty,amsthm,amssymb,fontenc,color,tensor,textcomp,amsfonts,relsize,graphicx,graphics,enumerate}
\usepackage{fancyhdr}
\usepackage[colorlinks=true,
            linkcolor=red,
            urlcolor=blue,
            citecolor=green]{hyperref}
\usepackage{mathtools}
\usepackage{amsmath,tikz}
\usepackage{tensor}
\usepackage[rmargin=1.5cm, lmargin=1.5cm]{geometry}
\newtheorem{theo}{Theorem}[section]
\newtheorem{lem}[theo]{Lemma}

\newtheorem{cor}[theo]{Corollary}
\newtheorem{df}[theo]{Definition}

\newtheoremstyle{rmk}% name
  {\topsep}% space above
  {\topsep}% space below
  {}% body font
  {}% indent amount
  {\itshape}% theorem head font
  {}% punctuation after theorem head
  {.5em}% space after theorem head
  {\thmname{#1}\thmnumber{ #2.}\thmnote{ (#3)}}% theorem head spec
\theoremstyle{rmk}
\newtheorem{rmk}[theo]{Remark}

\makeatletter
\newcommand{\address}[1]{\gdef\@address{#1}}
\newcommand{\email}[1]{\gdef\@email{\url{#1}}}
\newcommand{\@endstuff}{\par\vspace{\baselineskip}\noindent\small
\begin{tabular}{@{}l}\scshape\@address\\{E-mail address:} \@email \end{tabular}}
\AtEndDocument{\@endstuff}
\makeatother

\title{ %Local Estimates of Hermite-Hadamard type in global nonpositive curvature spaces OR 
Fractional Integral Estimates of Hermite-Hadamard type in  Global Nonpositive Curvature Spaces
}
\date{}
\author{
Peter Olamide Olanipekun }
\address{Department of Mathematics, The University of Auckland, Auckland 1010, New Zealand.\\ E-mail address: \color{blue} olanipekunp@gmail.com}

\email{peter.olanipekun@auckland.ac.nz}
\date{}
\begin{document}\sloppy
\maketitle

\abstract{We extend the notion of convexity of functions defined on global nonpositive curvature spaces by introducing (geodesically) $h$-convex functions. We prove estimates of Hermite-Hadamard type via Katugampola's fractional integrals. We obtain an important  corollary which gives an essentially sharp estimate involving squared distance mappings between points in a global NPC space. This is a contribution to analysis on spaces with curved geometry.}

\section{Introduction}

\iffalse

IMRN, Bulletin Belgian Math Soc., Nagoya Math Soc., Applied Maths Letters, Proc. Amer. Soc., Nonlinear Analysis Theory and Applications SCience direct
Analysis and Geometry in metric space
Israel OR Ukranian Journal of mathematics
Differential Geometry - Dynamical Systems (journal :))

Intro with prelim towards the end like in fractal paper. State results and corallaries

Proof of theorems 

do Corollary 3.5 in Conde version of HH ineq.

see  https://journals.ontu.edu.ua/index.php/geometry/article/view/1605/1811

%https://en.wikipedia.org/wiki/Geodesic_convexity#:~:text=In%20mathematics%20%E2%80%94%20specifically%2C%20in%20Riemannian,of%20a%20set%20or%20function.
Re write your definiton of geodesic convexity using the link above and other materials you might find, to include not just zero and 1 but generally intervals a and b.

https://link.springer.com/article/10.1007/BF00940467

Use https://arxiv.org/pdf/1806.06373.pdf PAGE 25 for the definition of totally convex set.

Check that you are not just repeating the phrase "Global nonpositive curvature spaces", mention a few times and start using the abbreviuation "global NPC spaces". That should work well. :)

This area is convex analysis.
\fi

A fundamental inequality satisfied by convex function is the Hermite-Hadamard inequality: let $f: [a,b]\subset\mathbb{R} \to \mathbb{R}$, then the following holds
\begin{align}
 f\left( \frac{a+b}{2}  \right)   \leq \frac{1}{b-a}\int_a^b f(x) dx \leq \frac{f(a)+ f(b)}{2}.     \label{hadaml}    
\end{align}

Note that the  inequalities above are reversed if $f$ is concave. 
The Hermite-Hadamard inequality also characterises convex functions defined on an interval of $\mathbb{R}$ \cite{bes}. Inequality \eqref{hadaml} was first proved in the article \cite{herm} by Hermite, and since then it has garnered a lot of attention in the literature, with notable improvements, extensions, generalisations, and refinements; see, for instance, the monographs \cite{dragomir, niculescu} and the references therein. There are questions on how to extend the Hermite-Hadamard inequality to metric spaces or Riemannian manifolds. Some authors have extended \eqref{hadaml} to functions of several variables. For instance, in \cite{drag1}, Dragomir established inequalities involving triple integrals for convex functions defined on a ball and established interesting properties for a certain convex mapping. He also established inequalities on a disk in $\mathbb{R}^2$ and derived results for mappings naturally connected to the inequalities he had established \cite{drag2}. By defining what it means for a function to be convex on the coordinates, the author also proved sharp inequalities for functions defined on rectangles with some interesting applications \cite{drag3}. A converse of \eqref{hadaml} for functions defined on simplices was proved in \cite{flavia}; the authors established that the Hermite-Hadamard inequality on simplices characterises convex functions under some conditions on the measure. In \cite{mih}, a generalisation of \eqref{hadaml} for convex functions defined on simplices is proven by using a volume formula and its higher-dimensional generalisation; this approach completely evades well-known tools from Choquet's theory (see \cite{phel}, \cite{nicu1} for more details). In \cite{nicu}, Niculescu extended Choquet's theorem to compact metric spaces with a global nonpositive curvature, and by using a Jensen-type inequality, he obtained a generalisation of \eqref{hadaml} to spaces with curved geometry. In such spaces, geodesics play the role of segments. By establishing a lemma which gives a unique minimal geodesic from a point to another on hemispheres, the author in \cite{barani1} proved an Hermite-Hadamard type inequality for integrable convex functions defined on hemispheres. There are a few studies of the Hermite-Hadamard inequality on nonpositively curved (NPC) spaces \cite{nicu, conde, conde1}. NPC spaces are fundamental in several areas of mathematics, especially geometry and topology.

Hadamard \cite{hadamard1} pioneered the study of what is now known as NPC spaces, and Cartan investigated generalisations of such spaces in higher dimensions. Later, the works of Alexandrov and Busemann now form much of the foundational works on the theory of metric spaces with upper curvature bounds \cite{alex1, alex2, bus1, bus2}.

A metric space $(M,d)$ has {\it nonpositive Alexandrov curvature} if for any $p\in M$ and any geodesic segment $\gamma_{[x,y]} \in M$ between points $x$ and $y$, the following inequality holds:
$$d^2(p, \gamma_{[x,y]}(1/2))   \leq \frac{1}{2} (d^2(p, x)   + d^2(p, y)) -\frac{1}{4} d^2(x,  y)$$
provided the points $x$ and $y$ are sufficiently close to $p$, and $\gamma_{[x,y]}(1/2)$ is the middle between $x$ and $y$, that is, $d(x, \gamma_{[x,y]}(1/2))= d(y,\gamma_{[x,y]}(1/2)) = \frac{1}{2} d(x,y)$. The inequality above is well known as the CN inequality of Bruhat and Tits \cite{bru}. Additionally, $(M,d)$ is called a {\it Hadamard space} if it is complete. Examples include simply connected complete Riemannian manifolds with negative constant curvature, Bruhat-Tits buildings, Hilbert spaces, the upper half-plane endowed with the Poincaré metric, see Section 1.6 in \cite{kiy} and the book \cite{brid} for more examples.

In this paper, we will extend the notion of convex functions defined on metric spaces with global nonpositive curvature. Thereafter, we will establish Hermite-Hadamard type inequalities for a general class of convex functions via Katugampola's fractional integrals (which we shall define in due course).

\section{Global NPC Spaces and  Notions of Convexity}
In this section, we recall some facts about global NPC spaces and convex functions on such spaces, details can be found in \cite{bus1, jost, conde, conde1, nicu}. We also introduce the concept of $h$-convexity on global NPC spaces.

\begin{df} \normalfont
Let $t_1, t_2 \in [0,1]$, a curve $\gamma$ is called {\it geodesic} if there exists $\varepsilon >0$ such that the length of $\gamma$, when restricted to $[t_1, t_2]$, is the metric distance between $\gamma(t_1)$ and $\gamma(t_2)$ provided that $|t_1-t_2| <\varepsilon$.

A metric space $(M, d)$ is called a  {\it geodesic space}, if for any two points $x,y \in M$, there exists a shortest geodesic arc joining them. In other words,  there is a continuous curve $\gamma:[0,1] \to M$ with endpoints $x=\gamma(0)$ and $y=\gamma(1)$ and the length of $\gamma$ is precisely the distance between the points $x$ and $y$.

 A geodesic space has a {\it global nonpositive curvature  in the sense of Busemann} if for any two shortest geodesics $\gamma, \tilde \gamma :[0,1] \to M$ with $\gamma(0)=x=\tilde\gamma(0)$, the distance map $t\mapsto d(\gamma(t), \tilde\gamma(t))$ is convex.  In other words, for every point $x,y,z\in M$, we have the inequality
\begin{align}
2d(\gamma_{[x,y]}(1/2) , \gamma_{[x,z]}(1/2)  ) \leq d(y,z)  \label{fgt1}
\end{align}
provided that $x, y$ are sufficiently close to $z$. If equality holds in \eqref{fgt1}, then we say that $(M,d)$ is flat. The space $(M,d)$ has {\it negative curvature in the sense of Busemann} if the inequality in \eqref{fgt1} is strict (this happens when the endpoint of neither geodesic is contained in the other one).

The space $(M,d)$  is called a {\it global NPC space} if the following conditions are satisfied
\begin{enumerate}
\item each pair of points can be connected by a  geodesic
\item for  $x_0, x_1\in M$ there exists a point $y\in M$ such that    for all $p\in M$
\begin{align}
d^2(p,y) \leq \frac{1}{2}( d^2(p,x_0) +d^2(p,x_1) ) -\frac{1}{4} d^2(x_0,x_1).
\end{align}

\end{enumerate}
\end{df}
Generally, the following comparison principle holds: let $\gamma_{[p, x_0]}$, $\gamma_{[x_0, x_1]}$ and $\gamma_{[p, x_1]}$ be three geodesic segments connecting the points $p, x_0, x_1 \in M$, and let $x_t$ be an arbitrary point on $\gamma_{[x_0, x_1]}$ which is a fraction of $d(x_0,x_1)$, then 
$$d(x_0, x_t)= td(x_0, x_1)\quad\quad\textnormal{and}\quad \quad d(x_t, x_1)= (1-t)d(x_0, x_1)$$
with the following inequality
$$d^2(p,x_t)\leq (1-t) d^2(p, x_0)  + t d^2(p, x_1)-t(1-t) d^2(x_0, x_1) \,, \quad\quad t\in[0,1].$$

\noindent
Denote $x_t:= (1-t) x_0 + tx_1$. Clearly, $x_{1/2}:= \gamma_{[x_0, x_1]}(1/2)$ is the midpoint of the  segment that connects $x_0$ and $x_1$, and the mean value of a function on $[0,1]$ exists, thus we can introduce the notion of convexity on global NPC spaces. We see at once that a function $f: K\subseteq M\to \mathbb{R}$ is convex if for all $t\in[0,1]$, we have $f(x_t)\leq (1-t)f(x_0) +tf(x_1)$.
Global NPC spaces have global nonpositive curvature in the sense of Busemann, they are also known as $\textnormal{CAT}(0)$ spaces.

\begin{theo}\cite{kor, res, uwe}
Let $(M,d)$ be a global NPC space. Let $x_0, x_1, y_0$ and $y_1$ be four points in $M$. Let $x_t$ be the point which is a fraction of $d(x_0, x_1)$. For any $t\in[0,1]$, the following holds
\begin{align}
d^2(x_t, y_0) + d^2(x_{1-t}, y_1) &\leq d^2(x_0,y_0) + d^2(x_1,y_1) +2t^2d^2(x_0, x_1)  \nonumber
\\& \quad\quad +t(d^2(y_0,y_1) - d^2(x_0,x_1)) -t(d(y_0, y_1) -d(x_0, x_1))^2.
\end{align}
\end{theo}

\begin{df}\normalfont
A subset $K\subseteq M$ is called convex if for each geodesic $\gamma:[0,1]\rightarrow M$ joining two arbitrary points in K, it holds that $\gamma([0,1])\subseteq K$.
\end{df}

\begin{df}\normalfont\label{srq}
A function $f:K\rightarrow \mathbb{R}$ is convex if the  function $f\circ \gamma:[0,1] \rightarrow \mathbb{R}$ is convex whenever $\gamma:[0,1]\rightarrow K$ is geodesic, that is, for all $t\in[0,1]$
\begin{align}
f(\gamma(t))\leq (1-t)f(x) +t f(y).   \label{jklq1}
\end{align}
\end{df}

Note that inequality \eqref{jklq1} follows from the convexity of $f\circ \gamma$, that is,

$$f(\gamma(t)) = f(\gamma[(1-t)\cdot 0 +t\cdot 1]) \leq (1-t) f\circ\gamma(0) +t f\circ \gamma(1) = (1-t)f(x) +t f(y).$$

In particular, the distance map $d:M \times M \to \mathbb{R}$ is convex. In other words, given any two geodesics $\gamma, \eta : [0,1] \to M$, we have the inequality
$$d( \gamma(t), \eta(t) )  \leq (1-t) d(\gamma(0),  \eta(0))  + td( \gamma(1), \eta(1) ). $$
This implies that every ball in a global NPC space is a convex set. Let $k>1$, for every $y\in M$,  the  map $G_y(x):= d^k(x,y)$   is strictly convex. That is, for every nonconstant geodesic $\gamma: [0,1] \to M$ and $t\in(0,1)$ we have the inequality
$G_y(\gamma(t)) < (1-t)G_y(\gamma(0)) + tG_y(\gamma(1))$.

If $-f$ is convex then $f$ is {\it concave}. The function $f$ is said to be {\it affine} if $f$ is both concave and convex.
There are several notions of convexity on metric spaces \cite{ohta, kell, shaikh}. 
Let $p>0$, a function is said to be $p$-convex if $f^p$ is convex. If  $f(\gamma(t)) \leq \textnormal{max}\{ f(\gamma(0)), f(\gamma(1)) \} $, then we say that $f$ is {\it quasi convex}. A function $f: K\to \mathbb{R}$ is {\it geodesically} $\varphi$-{\it convex} if there is a function $\varphi: \mathbb{R}\times \mathbb{R} \to \mathbb{R}$ such that $f(\gamma(t))  \leq f(x)  +t\varphi( f(y), f(x))$ for all $x,y \in K$ and $t\in[0,1]$. For example, let $M= \mathbb{R} \times \mathbb{S}^1$ and $\varphi(x,y):= x^3-y^3$, the function  $f:K \subset M\to \mathbb{R}$ defined by $f(s,c)= x^3$ is geodesically $\varphi$-convex but not convex. The notion of geodesically invex sets and geodesically pre-invex functions can be similarly defined on Riemannian manifolds (see \cite{bar}).

%There are articles on notions of convexities on metric spaces \cite{ohta} ("L-convexity which can be seen as a relaxed form of Busemann's nonpositive curvature assumption") and \cite{shaikh} . Use \quad
%\cite{https://sites.pitt.edu/~manfredi/papers/convexCofV.pdf}.   convex functions on Heisenberg groups
%other notions: https://arxiv.org/pdf/1601.03363.pdf
%\cite{kell}

The notion of $h$-convexity for functions defined on an interval of $\mathbb{R}$ was introduced by Varo\u sanec in \cite{varo}. It is known that the class of $h$-convex functions unifies existing classes of convex functions such as $s$-convex functions, Godunova-Levin functions, and $P$-functions. Motivated by the results in \cite{varo, conde1}, we extend the notion of $h$-convex functions to global NPC spaces.

\begin{df}\normalfont  \label{h-conv}
Let $K\subseteq M$ be a convex subset of a global NPC space, and let $h: \mathbb{R} \to (0, \infty).$ A function $f:K\rightarrow \mathbb{R}$ is geodesically $h$-convex if the function $f\circ \gamma:[0,1] \rightarrow \mathbb{R}$ is $h$-convex whenever $\gamma:[0,1]\rightarrow K$ is geodesic, that is, for all $t\in[0,1]$ we have
$$f(\gamma(t)) \leq h(1-t)f(\gamma(0)) +h(t) f(\gamma(1)).$$
\end{df}

\begin{rmk}
Observe that if $h(t) \geq t$ for all $t\in[0,1]$, then non-negative (geodesically) convex functions are $h$-convex.
\end{rmk}

Let $h_k(x)= x^k$, $x>0$. It is known that the function $g(x)= x^r$ where $x>0$, is $h_k$-convex if $r\in (-\infty,0]\cup [1, \infty)$ and $k\leq 1$. Also, $g$ is $h_k$-convex if $r\in(0,1)$ and $k\leq r$. If we define $f\circ\gamma:= g$ and restrict the domain of $g$ to $[0,1]$, then $f$ is geodesically $h_k$-convex on $K\subseteq M$ if $g$ is $h_k$-convex on $[0,1]$.

\begin{theo}\cite{sak} \label{sak1}
Let $f:[a,b] \rightarrow\mathbb{R}$ be  an $h$-convex function, then
\begin{align}
\frac{1}{2h\left(\frac{1}{2}  \right)}  f\left(\frac{a+b}{2}  \right) \leq \frac{1}{b-a} \int_a^b f(x) dx \leq [f(a) +f(b)]\int_0^1 h(t) dt.  \label{hfg7}
\end{align}
\end{theo}

Theorem \eqref{sak1}  is known as the Hermite-Hadamard inequality for $h$-convex functions. With $h(x)=x$, the inequality \eqref{hfg7} reduces to \eqref{hadaml}.

The following lemma shows some properties of geodesics and convex functions in a global NPC space.
\begin{lem}\cite{conde1}
Let $(M,d)$ be a global NPC space, $K\subseteq M$ a convex set and $\gamma: [0,1] \to K$ a geodesic connecting $\gamma(0)$ and $\gamma(1)$. Then
\begin{enumerate}
\item For $t_1, t_2\in[0,1]$ the curve $\gamma\big|_{[t_1, t_2]} (\lambda) = \gamma((1-\lambda)t_1 +\lambda t_2)$ is the unique geodesic connecting $\gamma(t_1)$ with $\gamma(t_2)$.

\item For any $t_0 \in [0,1]$ the midpoint between $\gamma(t_0)$ and $\gamma(1-t_0)$ is given by $\gamma(1/2)$.

\item  If $f:K \to\mathbb{R}$ is convex, then $\int_0^1 f(\gamma(u)) du =\int_0^1 f(\gamma(1-t)) dt$.
\end{enumerate}

\end{lem}
With the help of the lemma above, Conde \cite{conde1} proved the following Hermite-Hadamard inequality for convex functions on global NPC space.

\begin{theo} \label{c09}
Let $(M,d)$ be a global NPC space, $K\subseteq M$ a convex subset and $f: K \to \mathbb{R}$  a convex function. Then 
$$f(\gamma(1/2)) \leq \int_0^1 f(\gamma(t)) dt \leq \frac{f(\gamma(0))  + f(\gamma(1))}{2}$$
 for all geodesic $\gamma:[0,1] \to K$.
\end{theo}

%\section{Main Results}

%\section{Properties of $h$-Geodesically Convex Functions}

%file:///C:/Users/pola722/Downloads/BF01837981.pdf

\section{Fractional Integral Inequalities for $h$-Geodesically Convex Functions}
The results contained in this section are extensions and generalisations of the Hermite-Hadamard inequality, and related results. Indeed, we obtain Hermite-Hadamard type inequalities for the class of $h$-geodesically convex functions, which naturally generalises the class of convex functions. Also, we employ Katugampola's fractional integral operators which are generalisations of the well known Riemann-Liouville integral operators and the Hadamard integral operators.

\begin{df}\normalfont \cite{pod}
Let $\alpha >0$ be such that $n-1<\alpha < n$, $n\in\mathbb{N}$. The  left and right sided Riemann-Liouvile fractional integrals of order $\alpha$ are given by 
$$J^\alpha_{a+} f(x):=   \frac{1}{\Gamma(\alpha)} \int_a^x (x-t)^{\alpha -1} f(t)dt$$
and 
$$J^\alpha_{b-} f(x):=   \frac{1}{\Gamma(\alpha)} \int_x^b (t-x)^{\alpha -1} f(t)dt$$
respectively, where $a< x< b$ and $\Gamma$ is the well known Euler's gamma function defined by $\Gamma(x):= \int_0^\infty t^{x-1} e^{-t} dt$.
\end{df}

\begin{df}\normalfont\cite{samk}
The left and right sided Hadamard fractional integrals of order $\alpha >0$ are defined by
$$H^\alpha_{a+}f(x):=  \frac{1}{\Gamma(\alpha)} \int_a^x \left( \ln \frac{x}{t} \right)^{\alpha-1} \frac{f(t)}{t} dt$$
and
$$H_{b-}^\alpha f(x):=   \frac{1}{\Gamma(\alpha)} \int_x^b \left(\ln \frac{t}{x}   \right)^{\alpha-1} \frac{f(t)}{t}  dt$$
\end{df}

\begin{df}\normalfont\cite{kat2}
Let $c\in\mathbb{R}$ and $1\leq p\leq \infty$. The space $X_c^p(a,b)$ is the set of all complex-valued Lebesgue measurable functions $f$ equipped with norm

\[ \|f\|_{X_c^p}=\begin{cases} 
      \left(  \int_a^b \frac{|t^c f(t)|^p}{t} dt  \right)^{\frac{1}{p}} &,\quad 1\leq p<\infty \\
      \textnormal{esssup}_{a\leq t\leq b} |t^c f(t)| &,\quad p=\infty. 
   \end{cases}
\]
The space $X_c^p(a,b)$ is the classical $L^p(a,b)$ space when $c=\frac{1}{p}$. 
\end{df}

\begin{df}\normalfont\cite{kat, kat2}
Let $[a,b]\subset \mathbb{R}$ be a finite interval. The left and right side Katugampola fractional integrals of order $\alpha >0$ of $f\in X_c^p (a,b)$ are defined by
$${}^\rho I_{a+}^\alpha f(x) :=\frac{\rho^{1-\alpha}}{\Gamma(\alpha)} \int_a^x \frac{t^{\rho-1}}{(x^\rho -t^\rho)^{1-\alpha}}f(t) dt $$
and 
$${}^\rho I_{b-}^\alpha f(x) :=\frac{\rho^{1-\alpha}}{\Gamma(\alpha)} \int_x^b \frac{t^{\rho-1}}{(t^\rho -x^\rho)^{1-\alpha}}f(t) dt $$
with $a< x< b$ and $\rho>0$, provided the integrals exists.
\end{df}

The fractional integral operators  ${}^\rho I_{a+}^\alpha $  and ${}^\rho I_{b-}^\alpha $  are well defined on ${X_c^p}(a,b)$ for $\rho \geq c$, as shown in \cite{kat2}. There is a relationship among the integral operators defined above. Let $\alpha >0$ and $\rho >0$, then for $x> a$, it can be shown \cite{kat2} that 
$$\lim_{\rho\to 1} {}^\rho I_{a+}^\alpha f(x) =  J_{a+}^\alpha f(x) \quad\quad\mbox{and}\quad\quad  \lim_{\rho\to 0} {}^\rho I_{a+}^\alpha f(x) =  H_{a+}^\alpha f(x).$$
Similar identities hold for the right sided integrals.

%In what follows, we assume that $f\in X_c^p(K)$.
\begin{theo}
Let $\alpha>0$ and $\rho>0$. Let $(M,d)$ be a global NPC space, $K\subseteq M$ a convex set and $f:K \rightarrow [0, \infty)$ a geodesic $h$-convex function with $h\in L^q[0,1]$, $q>1$. Then the following inequalities hold
\begin{align}
f\left( \gamma\Big|_{[a^\rho , b^\rho] }\left( \frac{1}{2} \right) \right) &  \leq \frac{\rho^\alpha \Gamma(\alpha+1)}{(b^\rho -a^\rho)^{\alpha}} h\left( \frac{1}{2} \right) \Big(    {}^\rho I^\alpha _{a+} f\left( \gamma\left( b^\rho \right)\right) +     {}^\rho I^\alpha _{b-} f\left( \gamma\left( a^\rho \right)\right)           \Big)    \nonumber
\\&\leq    h\left( \frac{1}{2} \right) [f(\gamma(a^\rho)) +f(\gamma(b^\rho))]   \left[ \alpha \left(\frac{q-1}{\alpha q-1} \right)^{\frac{q-1}{q}}  \|h\|_{L^q[0,1]} + \rho^\alpha\Gamma(\alpha+1)  \,\,{}^\rho I_{0+}^\alpha h(1)   \right]                                   \label{cb1}
\end{align}
where $0\leq a, b\leq 1$.
\end{theo}

\begin{proof}
First, note that since $f$ is geodesically $h$-convex, we have
$$f\left( \gamma\left( \frac{x^\rho + y^\rho }{2} \right) \right)   \leq h\left(  \frac{1}{2}\right) [f(\gamma(x^\rho)) +f(\gamma(y^\rho))].$$

\noindent
Using this fact, and   the change of variables $x^\rho=t^\rho a^\rho +(1-t^\rho)b^\rho$ and $y^\rho=t^\rho b^\rho +(1-t^\rho)a^\rho$, we find for all
$t\in[0,1]$ and $0\leq a\leq x,y\leq b \leq 1$ the estimate
\begin{align}
f\left( \gamma\Big|_{[a^\rho , b^\rho] }\left( \frac{1}{2} \right) \right)= f\left( \gamma\left( \frac{x^\rho + y^\rho }{2} \right) \right) \leq h\left(  \frac{1}{2}\right) [f(\gamma(t^\rho a^\rho +(1-t^\rho)b^\rho)) +f(\gamma(t^\rho b^\rho +(1-t^\rho)a^\rho))] .    \nonumber
\end{align}

Multiplying the latter by $t^{\alpha\rho-1}$ and integrating over $t\in[0,1]$ yields
\begin{align}
\frac{1}{\alpha\rho}f\left( \gamma\Big|_{[a^\rho , b^\rho] }\left( \frac{1}{2} \right) \right)    &\leq h\left( \frac{1}{2} \right) \int_a^b \left( \frac{b^\rho-x^\rho}{b^\rho-a^\rho}  \right)^{\alpha-1}\frac{x^{\rho-1}}{b^\rho-a^\rho}    f(\gamma(x^\rho))   dx   \nonumber
\\&\quad\quad  + h\left( \frac{1}{2} \right) \int_a^b \left( \frac{y^\rho-a^\rho}{b^\rho-a^\rho}  \right)^{\alpha-1}\frac{y^{\rho-1}}{b^\rho-a^\rho}    f(\gamma(y^\rho))   dy           \nonumber
\\&=   \frac{\rho^{\alpha-1} \Gamma(\alpha)}{(b^\rho -a^\rho)^{\alpha}} h\left( \frac{1}{2} \right) \Big(    {}^\rho I^\alpha _{a+} f\left( \gamma\left( b^\rho \right)\right) +     {}^\rho I^\alpha _{b-} f\left( \gamma\left( a^\rho \right)\right)           \Big)     \nonumber
\end{align}
which proves the first inequality in \eqref{cb1}.
To prove the second inequality, we use  the geodesic $h$-convexity of $f$ to obtain
\begin{align}
&f(\gamma(t^\rho a^\rho +(1-t^\rho)b^\rho)) +f(\gamma(t^\rho b^\rho +(1-t^\rho)a^\rho))                     \nonumber
\\ &\leq    h(t^\rho) f(\gamma(a^\rho))    +  h(1-t^\rho) f(\gamma(b^\rho))       +  h(t^\rho) f(\gamma(b^\rho))              + h(1-t^\rho) f(\gamma(a^\rho))  .   \nonumber
\end{align}
Multiplying both sides by $ h\left( \frac{1}{2} \right)\alpha \rho t^{\alpha \rho -1}$, and integrating over $[0,1]$ with respect to $t$ yields
\begin{align}
 \frac{\rho^{\alpha} \Gamma(\alpha+1)}{(b^\rho -a^\rho)^{\alpha}} h\left( \frac{1}{2} \right) &\Big(    {}^\rho I^\alpha _{a+} f\left( \gamma\left( b^\rho \right)\right) +     {}^\rho I^\alpha _{b-} f\left( \gamma\left( a^\rho \right)\right)           \Big)  \nonumber
\\&\leq   h\left( \frac{1}{2} \right) \alpha \rho [f(\gamma(a^\rho)) +f(\gamma(b^\rho))]  \int_0^1           t^{\alpha \rho -1} (h(t^\rho)  +h(1-t^\rho))  dt.    \label{cnw}
\end{align}

For all $t\in[0,1]$ and $q>1$, we use H\"older  inequality to find
\begin{align}
\int_0^1           t^{\alpha \rho -1} h(t^\rho) dt \leq \frac{1}{\rho} \left(\frac{q-1}{\alpha q-1} \right)^{\frac{q-1}{q}}  \|h\|_{L^q[0,1]} .  \label{hol}
\end{align}

%Note that the change of variable $u= t^\rho$ yields
%\begin{align}
%\int_0^1 t^{\alpha\rho-1} h(t^\rho) dt = \int_0^1 u^{\alpha-1} h(u) du= \Gamma(\alpha) J^\alpha_{1-} h(0).  \label{hol}
%\end{align}

On the other hand, we use the change of variable $u^\rho= 1-t^\rho$ to find 
\begin{align}
\int_0^1           t^{\alpha \rho -1} h(1-t^\rho) dt =\int_0^1 (1-u^\rho)^{\alpha-1} u^{\rho-1} h(u^\rho) du   = \frac{\Gamma(\alpha)}{\rho^{1-\alpha}}  {}^\rho I_{0+}^\alpha h(1).             \label{jhk}
\end{align}

Noting that $h$ is nonnegative by definition, we combine \eqref{jhk}, \eqref{hol} and \eqref{cnw}
to prove the second inequality in \eqref{cb1}.  This completes the proof.

\end{proof}

\begin{rmk}\normalfont
Note that the change of variable $u= t^\rho$ yields
 $$\Gamma(\alpha) J^\alpha_{1-} h(0)= \frac{1}{\rho}\int_0^1 u^{\alpha-1} h(u) du = \int_0^1 t^{\alpha\rho-1} h(t^\rho) dt   \overset{\eqref{hol}}{\leq}  \frac{1}{\rho} \left(\frac{q-1}{\alpha q-1} \right)^{\frac{q-1}{q}}  \|h\|_{L^q[0,1]}. $$
Thus we can remove the condition that $h\in L^q[0,1]$ and refine the estimate \eqref{cb1}. We have the following theorem.
\end{rmk}

\begin{theo}
Let $\alpha>0$ and $\rho>0$. Let $(M,d)$ be a global NPC space, $K\subseteq M$ a convex set and $f:K \rightarrow\mathbb{R}$ a geodesic $h$-convex function. Then the following inequalities hold
\begin{align}
f\left( \gamma\Big|_{[a^\rho , b^\rho] }\left( \frac{1}{2} \right) \right) &  \leq \frac{\rho^\alpha \Gamma(\alpha+1)}{(b^\rho -a^\rho)^{\alpha}} h\left( \frac{1}{2} \right) \Big(    {}^\rho I^\alpha _{a+} f\left( \gamma\left( b^\rho \right)\right) +     {}^\rho I^\alpha _{b-} f\left( \gamma\left( a^\rho \right)\right)           \Big)    \nonumber
\\&\leq    h\left( \frac{1}{2} \right) [f(\gamma(a^\rho)) +f(\gamma(b^\rho))]  \Gamma(\alpha+1) \left[ \rho J^\alpha_{1-} h(0)              + \rho^\alpha  \,\,{}^\rho I_{0+}^\alpha h(1)   \right]                                 \label{cb2}
\end{align}
where $0\leq a, b\leq 1$.
\end{theo}

\begin{theo}\label{lab}
Let $\alpha>0$ and $\rho>0$. Let $(M,d)$ be a global NPC space, $K\subseteq M$ a convex set and $f:K \rightarrow\mathbb{R}$ a geodesic $h$-convex function. Then the following inequalities hold
\begin{align}
f\left( \gamma\left( \frac{1 }{2} \right) \right) & 
 \leq \frac{\rho^\alpha \Gamma(\alpha+1)}{(b^\rho -a^\rho)^{\alpha}} h\left( \frac{1}{2} \right) \Big(    {}^\rho I^\alpha_{a+} f(\gamma(b^\rho))                        +                 {}^\rho I^\alpha_{c-} f(\gamma(s^\rho))                \Big)    \nonumber
\\&\leq   \frac{f(\gamma(0))      +  f(\gamma(1)) }{(b^\rho-a^\rho)^\alpha}   \rho^{\alpha}    \Gamma(\alpha+1) h\left( \frac{1}{2} \right)  \Big(    {}^\rho I^\alpha_{a+} h(b^\rho)                        +                 {}^\rho I^\alpha_{c-} h(s^\rho)              \Big)          \label{ty1}                         
\end{align}
where $c:= (1-a^\rho)^{\frac{1}{\rho}}$, $s:=(1-b^\rho)^{\frac{1}{\rho}} $ and $0\leq a, b\leq 1$.
\end{theo}

\begin{proof}
Since $f$ is geodesically $h$-convex, we have
\begin{align}
f\left(\gamma\left(1/2\right)  \right) \leq h\left(1/2 \right) (f(\gamma(x^\rho))  + f(\gamma(1-x^\rho))  )   . \label{hj1}
\end{align}

Setting $x^\rho= t^\rho a^\rho +(1-t^\rho) b^\rho$ where $t\in[0,1]$, multiplying both sides of \eqref{hj1} by  $t^{\alpha\rho-1}$ and integrating over $[0,1]$ gives
\begin{align}
\frac{f(\gamma(1/2))}{\alpha\rho}            \nonumber
&\leq   h(1/2) \int_0^1 t^{\alpha\rho-1} \Big(  f(\gamma(t^\rho a^\rho +(1-t^\rho)b^\rho))    +  f(\gamma(   (1-b^\rho) +t^\rho(b^\rho-a^\rho)))                     \Big)dt								\nonumber 
\\&= h(1/2) \int_a^b \left( \frac{b^\rho-x^\rho}{b^\rho-a^\rho} \right)^{\alpha-1} \frac{x^{\rho-1}}{b^\rho-a^\rho}  f(\gamma(x^\rho)) dx	\nonumber
\\&\quad\quad	+   h(1/2) \int_{(1-b^\rho)^{\frac{1}{\rho}}}^{(1-a^\rho)^{\frac{1}{\rho}}}      \left( \frac{u^\rho-(1-b^\rho)}{b^\rho-a^\rho} \right)^{\alpha-1} \frac{u^{\rho-1}}{b^\rho-a^\rho}  f(\gamma(u^\rho)) du						\nonumber
\\&=  \frac{h(1/2)}{(b^\rho -a^\rho)^\alpha} \Gamma(\alpha) \rho^{\alpha-1}  \Big(    {}^\rho I^\alpha_{a+} f(\gamma(b^\rho))                        +                 {}^\rho I^\alpha_{c-} f(\gamma(s^\rho))              \Big). \label{fg}
\end{align}

 The first inequality in \eqref{ty1} follows from \eqref{fg}. Next is to prove the second inequality in \eqref{ty1}. Since $f$ is geodesically $h$ convex, we have
\begin{align}
&f(\gamma(t^\rho a^\rho +(1-t^\rho)b^\rho))    +  f(\gamma(   (1-b^\rho) +t^\rho(b^\rho-a^\rho)))   \nonumber
\\&=f(\gamma(x^\rho)) + f(\gamma(1-x^\rho))    \nonumber
\\&\leq [h(x^\rho) +h(1-x^\rho)] [f(\gamma(0))      +  f(\gamma(1)) ]                               \nonumber
\\& =     [h(t^\rho a^\rho +(1-t^\rho)b^\rho) +h((1-b^\rho) +t^\rho(b^\rho-a^\rho))] [f(\gamma(0))      +  f(\gamma(1)) ]   .    \label{ty2}
\end{align}
Multiplying both sides of \eqref{ty2} by $t^{\alpha\rho-1}$ and integrating over $[0,1]$, we have
\begin{align}
\frac{     \Gamma(\alpha) \rho^{\alpha-1}}{(b^\rho -a^\rho)^\alpha}  \Big(    {}^\rho I^\alpha_{a+} f(\gamma(b^\rho))                        +                 {}^\rho I^\alpha_{c-} f(\gamma(s^\rho))              \Big) \leq
\frac{f(\gamma(0))      +  f(\gamma(1)) }{(b^\rho-a^\rho)^\alpha}       \Gamma(\alpha) \rho^{\alpha-1}  \Big(    {}^\rho I^\alpha_{a+} h(b^\rho)                        +                 {}^\rho I^\alpha_{c-} h(s^\rho)              \Big)    \label{gh2}
\end{align}
The second inequality in \eqref{ty1} follows from \eqref{gh2}. This completes the proof.
\end{proof}

Let $k\geq 1$, recall from Definition \ref{srq} that the function $G_y(x):= d^k(x,y)$ is convex. Let  $h: [0,1]\to(0, \infty)$ be a map satisfying $h(t)\geq t$ for all $t\in[0,1]$, then
\begin{align}
G_y(\gamma(t)) \leq (1-t)G_y(\gamma(0)) +tG_y(\gamma(1)) \leq h(1-t) G_y(\gamma(0)) +h(t) G_y(\gamma(1))   \nonumber
\end{align}
so that $G_y$ is geodesically $h$-convex. In particular, the function $g_y(t)= d^k (y, \gamma_{[x_1, x_2]}(t)) $ is $h$-convex.
Consequently, by Theorem \ref{lab}, we have
\begin{align}
d^k(y, \gamma_{[x_1, x_2]}(1/2))& 
 \leq \frac{\rho^\alpha \Gamma(\alpha+1)}{(b^\rho -a^\rho)^{\alpha}} h\left( \frac{1}{2} \right) \Big(    {}^\rho I^\alpha_{a+} d^k(y, \gamma_{[x_1, x_2]}(b^\rho))                        +                 {}^\rho I^\alpha_{c-}d^k(y, \gamma_{[x_1, x_2]}(s^\rho))                \Big)    \nonumber
\\&\leq   \frac{d^k(y, x_1)      +  d^k(y, x_2)  }{(b^\rho-a^\rho)^\alpha}   \rho^{\alpha}    \Gamma(\alpha+1) h\left( \frac{1}{2} \right)  \Big(    {}^\rho I^\alpha_{a+} h(b^\rho)                        +                 {}^\rho I^\alpha_{c-} h( s^\rho)              \Big)          \nonumber                         
\end{align}
where $c:= (1-a^\rho)^{\frac{1}{\rho}}$, $s:=(1-b^\rho)^{\frac{1}{\rho}} $ and $0\leq a, b\leq 1$. 
By assuming that $h(t)\geq t$ for all  $t\in[0,1]$, we use the $h$-convexity of $g:[0,1] \to \mathbb{R}$, $g(t)= d^2(\gamma(t), \tilde\gamma(t) )$ (where $\gamma$ and $\tilde\gamma$ are geodesics) to obtain the following corollary, for $k=2$.

\begin{cor}
Let $\alpha>0$ and $\rho>0$. Let $(M,d)$ be a global NPC space, and let $\gamma:=\gamma_{[x_1, x_2]}$ and $\tilde\gamma:=\tilde\gamma_{[y_1, y_2]}$ be two geodesics connecting the points  $x_1, x_2 \in M$ and $y_1,y_2\in M$ respectively. Suppose that $h:[0,1] \to (0,\infty)$ is a function satisfying $h(t)\geq t$, for all $t\in[0,1]$. Then the following inequalities hold
\begin{align}
d^2( \tilde\gamma(1/2), \gamma(1/2))& 
 \leq \frac{\rho^\alpha \Gamma(\alpha+1)}{(b^\rho -a^\rho)^{\alpha}} h\left( \frac{1}{2} \right) \Big(    {}^\rho I^\alpha_{a+} d^2( \tilde\gamma(b^\rho), \gamma(b^\rho))                        +                 {}^\rho I^\alpha_{c-}d^2(\tilde\gamma(s^\rho), \gamma(s^\rho))                \Big)    \nonumber
\\&\leq  \frac{d^2(y_1,x_1)      +  d^2(y_2, x_2)  }{(b^\rho-a^\rho)^\alpha}   \mathcal{E}(h)        -              C(\alpha,\rho)[  d(y_1, y_2) d(x_1, x_2)             ]^2   \nonumber
\\&\leq   \frac{d^2(y_1,x_1)      +  d^2(y_2, x_2)  }{(b^\rho-a^\rho)^\alpha}   \mathcal{E}(h)    \nonumber                          
\end{align}
where $c:= (1-a^\rho)^{\frac{1}{\rho}}$, $s:=(1-b^\rho)^{\frac{1}{\rho}} $ , $0\leq a, b\leq 1$, $C(\alpha, \rho) \geq 0$,
\[
C(\alpha, \rho):=  \frac{   (a^\rho\alpha+b^\rho)(2(\alpha+2)-4b^\rho)-2a^{2\rho} \alpha(\alpha+1)}{\alpha\rho(\alpha+1)(\alpha+2)}  \quad\textnormal{and} \quad     \mathcal{E}(h):= \rho^{\alpha}    \Gamma(\alpha+1) h\left( \frac{1}{2} \right)  \Big(    {}^\rho I^\alpha_{a+} h(b^\rho)                        +                 {}^\rho I^\alpha_{c-} h(s^\rho)              \Big)  .   
\]
\end{cor}

\begin{proof}
We use Corollary 2.5 in \cite{sturm} (a geodesic comparison result)   to write the estimate
\begin{align}
d^2( \tilde\gamma_{[y_1,y_2]}(u^\rho), \gamma_{[x_1, x_2]} (u^\rho)) &\leq h(1-u^\rho) d^2(y_1,x_1 ) + h(u^\rho) d^2(y_2, x_2) -u^\rho(1-u^\rho)[ d(y_1, y_2) -d(x_1, x_2)  ]^2  \nonumber
\\&\leq  h(1-u^\rho) d^2(y_1,x_1 ) + h(u^\rho) d^2(y_2, x_2) \nonumber
\end{align}
where $u\in[0,1]$ and $\rho>0$.  Hence, we deduce that
\begin{align}
&d^2( \tilde\gamma_{[y_1,y_2]}(u^\rho), \gamma_{[x_1, x_2]} (u^\rho)) + d^2( \tilde\gamma_{[y_1,y_2]}(1-u^\rho), \gamma_{[x_1, x_2]} (1-u^\rho)) \nonumber
\\ &\leq [h(u^\rho)+h(1-u^\rho)] d^2(y_1,x_1 ) + [h(u^\rho)+h(1-u^\rho)] d^2(y_2, x_2) -2u^\rho(1-u^\rho)[ d(y_1, y_2) -d(x_1, x_2)  ]^2  \nonumber
\\&\leq  [h(u^\rho)+h(1-u^\rho)] \left[d^2(y_1,x_1 ) +  d^2(y_2, x_2)  \right] . \label{rstu}
\end{align}

Set $u^\rho= t^\rho a^\rho +(1-t^\rho) b^\rho$  in the estimates in \eqref{rstu}, so that $u^\rho$ is the line segment connecting $a^\rho$ and $b^\rho$. Next, multiply the resulting estimates  by $h(1/2) t^{\alpha\rho-1}$ and   integrate on $[0,1]$ with respect to $t$, via Katugampola's fractional integral operators. Finally, apply Theorem \ref{lab}. 
Note that the constant $C(\alpha, \rho)$ is computed as follows.

Since $0\leq u^\rho \leq 1$, we have $2u^\rho(1-u^\rho) \geq 0$. Hence
\begin{align}
0\leq C(\alpha, \rho)=&\int_0^1[2(  t^\rho a^\rho +(1-t^\rho) b^\rho )-2  ( t^\rho a^\rho +(1-t^\rho) b^\rho)^2 ] t^{\alpha\rho-1} dt  \nonumber
=- \frac{2(b^\rho-a^\rho)}{\rho(\alpha+1)}  +\frac{2b^\rho}{\alpha\rho}    \nonumber
\\&- \frac{2a^{2\rho}\alpha(\alpha+1) +  2b^{2\rho}[(\alpha+1)(\alpha+2) -2\alpha(\alpha+2) +\alpha(\alpha+1) ]  +4a^\rho b^\rho[\alpha(\alpha+2) -\alpha(\alpha+1)  ]          }{\alpha\rho(\alpha+1)(\alpha+2)}    \nonumber
\\&= -\frac{2(b^\rho-a^\rho)\alpha(\alpha+2)-2b^\rho(\alpha+1)(\alpha+2)            }{\alpha\rho(\alpha+1)(\alpha+2)}       
          -   \frac{   2a^{2\rho}\alpha(\alpha+1) +  4b^{2\rho}  +4a^\rho b^\rho\alpha           }{\alpha\rho(\alpha+1)(\alpha+2)}                  \nonumber
\\&=  -\frac{2a^{2\rho} \alpha(\alpha+1)   +(a^\rho\alpha+b^\rho)(4b^\rho-2(\alpha+2))}{\alpha\rho(\alpha+1)(\alpha+2)}   .               \nonumber
\end{align}

This concludes the proof.
\end{proof}

 \bibliographystyle{amsplain}

\end{document}